\documentclass[a4paper,10pt]{article}
\title{Critical points of random branched coverings of the Riemann sphere}
\author{Michele Ancona \thanks{Institut Camille Jordan, Umr Cnrs 5208, Universit\'{e} Claude Bernard Lyon 1. ancona@math.univ-lyon1.fr} }
\date{}
\usepackage[a4paper,bindingoffset=0.2in,%
            left=1in,right=1in,top=1in,bottom=1in,%
            footskip=.25in]{geometry}
\usepackage[latin1]{inputenc}
\usepackage[T1]{fontenc}
\usepackage[english]{babel}
\usepackage{amsmath,amssymb,amsthm}
\usepackage{amsfonts}%
\usepackage{amssymb}%
\usepackage{indentfirst}
\usepackage{fancyhdr}
\usepackage{physics}
\usepackage{young}
\usepackage{graphicx}

\usepackage[vcentermath]{youngtab}

\RequirePackage[colorlinks,linkcolor=blue,citecolor=blue,urlcolor=blue]{hyperref}
\usepackage{enumerate}
\usepackage{color}
\usepackage{pst-fill,pst-grad,pst-plot,pst-eucl,pstricks-add,pst-node}
\usepackage{mathdots}
\usepackage{dsfont}

\theoremstyle{plain}
\newtheorem{thm}{Theorem}[section]

\newtheorem{prop}[thm]{Proposition}

\newtheorem{oss}[thm]{Remark}

\theoremstyle{definition}
\newtheorem{defn}[thm]{Definition}

\newcommand{\R}{\mathbb{R}}
\newcommand{\C}{\mathbb{C}}

\newcommand{\Crit}{\textrm{Crit}}
\newcommand{\Vol}{\textrm{Vol}}

\newcommand{\Pic}{\textrm{Pic}}
\newcommand{\dH}{\textrm{dH}}
\newcommand{\M}{\mathcal{M}}

\allowdisplaybreaks

\begin{document}
\maketitle
\begin{abstract}
Given a closed Riemann surface $\Sigma$ equipped with a volume form $\omega$, we construct a  natural probability measure on the space $\mathcal{M}_d(\Sigma)$ of degree $d$ branched coverings from $\Sigma$ to the Riemann sphere $\C\mathbb{P}^1.$  We prove a large deviations principle for the number of critical points in a given open set $U\subset \Sigma$: given any sequence $\epsilon_d$ of positive numbers, the probability that the number of critical points of a branched covering deviates from $2d\cdot\Vol(U)$ more than $\epsilon_d\cdot d$ is smaller than $\exp(-C_U\epsilon^3_d d)$, for some positive constant $C_U$. In particular, the probability that a covering does not have any critical point in a given open set goes  to zero exponential fast with the degree.
\end{abstract}

\section*{Introduction}
This paper is concerned with the branched coverings $u:\Sigma\rightarrow\C\mathbb{P}^1$ of very large degree from a closed Riemann surface $\Sigma$ to the Riemann sphere.
By the Riemann-Hurwitz formula, the number   of critical  points of such maps, counted with multiplicity,   equals $\#\Crit(u)=2d+2g-2$, where $g$ denotes the genus of $\Sigma$ and $d$  is the degree of the map.
\begin{center}\emph{ How do  these $2d+2g-2$ critical points distribute on $\Sigma$, if we pick   $u:\Sigma\rightarrow \C \mathbb{P}^1$ at random?}
\end{center}
In order to answer the question, we first construct a  probability measure on the space $\mathcal{M}_d(\Sigma)$ of  degree $d$ branched coverings $u:\Sigma\rightarrow \C \mathbb{P}^1$.
This probability measure is denoted by  $\mu_d$ and it is associated with a volume form $\omega$ on $\Sigma$ of total mass $1$ (that is $\int_{\Sigma}\omega=1$), which is fixed once for all. 
Later in the introduction we will sketch the construction of the measure $\mu_d$, which we will give in details in Section \ref{sectproba}.\\
The distribution of the critical points of a map $u\in\mathcal{M}_d(\Sigma)$ is encoded by the associated empirical measure which we renormalize by $2d+2g-2$, so that  its mass does  not  depend on $d\in\mathbb{N}^*$. More precisely, for any degree $d$  branched coverings $u\in\mathcal{M}_d(\Sigma)$, we consider the probability measure $T_u$ on $\Sigma$  defined by $$T_u=\frac{1}{2d+2g-2}\displaystyle\sum_{x\in\Crit(u)}\delta_x$$ where $\delta_x$ stands for the Dirac measure at $x$.
% The purpose of this paper is to answer the following question: 
The central object of the paper is then  the random variable 
$u\in (\mathcal{M}_d(\Sigma),\mu_d)\mapsto T_u\in\textrm{Prob}(\Sigma)$ which takes values in the space $\textrm{Prob}(\Sigma)$ of probabilities on $\Sigma$. 
The expected value $\mathbb{E}[T_u]$ of $T_u$ converges in the weak topology to the volume form $\omega$ of $\Sigma$, see Theorem \ref{expected}. It means that, for any continuous function $f:\Sigma\rightarrow\R$, one has $$\displaystyle\mathbb{E}_d[T_u(f)]\doteqdot\int_{u\in\mathcal{M}_d(\Sigma)}T_u(f)\textrm{d}\mu_d(u)\xrightarrow[d \rightarrow \infty]{} \int_{\Sigma}f\omega.$$
The main theorem of the paper is the following large deviations estimate for the random variable $T_u$.
\begin{thm}\label{smooth} Let $\Sigma$ be a closed Riemann surface equipped with a  volume form $\omega$ of mass $1$. For any smooth function  $f\in \mathcal{C}^{\infty}(\Sigma,\R)$ and any sequence $\epsilon_d$ of positive real numbers of the form $\epsilon_d=O(d^{-a})$, for some $a\in [0,1)$, there exists a positive constant $C$ such that the following inequality $$\mu_d\bigg\{u\in\M_d(\Sigma), \big|T_u(f)-\int_{\Sigma}f\omega\big|\geq\epsilon_d \bigg\}\leq \exp(-C\epsilon_dd)$$
holds.
\end{thm}
The key point in the proof of Theorem \ref{smooth} is a large deviations estimate for the $L^1$-norm of the random variable $u\in\mathcal{M}_d(\Sigma)\mapsto |\log \norm{du}|$, see Proposition \ref{fondame}. This estimate is obtained  by combining  H\"ormander peak sections  and some properties of subharmonic functions.\\
It turns out that the constant $C$ in Theorem \ref{smooth} is of the form $\frac{C'}{\norm{\partial\bar{\partial}f}_{\infty}}$, where $C'$ is a constant which does not depend on $f$, but  only on the sequence $\epsilon_d$ (and on $\Sigma$).\\
One of the consequence of Theorem \ref{smooth} is the following large deviations estimate for overcrowding and undercrowding of critical points in a given open set $U\subset\Sigma$.
\begin{thm}\label{number} Let $\Sigma$ be a closed Riemann surface equipped with a  volume form $\omega$ of mass $1$. For any open subset $U\subset \Sigma$ with  $\mathcal{C}^2$ boundary, there exists a positive constant $C_U$ such that, for  any sequence $\epsilon_d\in\R_+$  of the form $\epsilon_d=O(d^{-a})$, for some $a\in [0,1)$, the following inequality $$\mu_d\bigg\{u\in\M_d(\Sigma), \big|\frac{1}{2d}\#(\Crit(u)\cap U)-\Vol(U)\big|\geq\epsilon_d \bigg\}\leq \exp(-C_U\epsilon_d^3 d)$$
holds.
\end{thm}
Theorem \ref{number} follows from Theorem \ref{smooth} by taking, as test functions, two sequences of functions $\psi^+_d$ and $\psi^-_d$ which approximate  from above and below, in an appropriate way, the characteristic function $\mathds{1}_U$ of the set $U$.\\
Another consequence of Theorem \ref{smooth} is the following hole probabilities result for critical points of random branched coverings.
\begin{thm}\label{hole} Let $\Sigma$ be a closed Riemann surface equipped with a  volume form $\omega$ of mass $1$. For every open subset $U\subset \Sigma$ there exists $C_U>0$ such that $$\mu_d\bigg\{u\in\M_d(\Sigma), \Crit(u)\cap U=\emptyset\bigg\}\leq \exp(-C_Ud).$$
\end{thm}
Remark that Theorem \ref{hole} is not a formal consequence of Theorem \ref{number} for the constant sequence $\epsilon_d\equiv\Vol(U)$. Indeed, in Theorem \ref{hole}, we do not require any regularity on the boundary of  $U$. \\

Let us briefly describe the construction of the probability measure $\mu_d$, which will be given in details in  Section \ref{sectproba}.
 The fundamental remark is that the space  $\mathcal{M}_d(\Sigma)$ is fibered over the space $\Pic^d(\Sigma)$ of degree $d$ line bundles on $\Sigma$. 
 Indeed, there is a  natural map from  $\mathcal{M}_d(\Sigma)$ to  $\Pic^d(\Sigma)$  which maps every morphism $u$ to the line bundle $u^*\mathcal{O}(1)$.  The fiber  of this map over $\mathcal{L}\in\Pic^d(\Sigma)$ is denoted by $\M_d(\Sigma,\mathcal{L})$. It is an open dense subset of $\mathbb{P}(H^0(\Sigma,\mathcal{L})^2)$ given by (the class of) pairs of global sections without common zeros.  In order to construct a probability measure on $\mathcal{M}_d(\Sigma)$, we produce a family of probability measures $\{\mu_{\mathcal{L}}\}_{\mathcal{L}\in\Pic^d(\Sigma)}$ on each space  $\mathbb{P}(H^0(\Sigma,\mathcal{L})^2)$ which restricts to a probability measure on $\M_d(\Sigma,\mathcal{L})$, still denoted by $\mu_{\mathcal{L}}$.  The probability measure $\mu_{\mathcal{L}}$ on $\mathbb{P}(H^0(\Sigma,\mathcal{L})^2)$ is the measure induced by the Fubini-Study metric associated with a Hermitian product on $H^0(\Sigma,\mathcal{L})^2$. This Hermitian product is a natural $L^2$-product induced by $\omega$, see Section \ref{bergman}. This family of measures, together with the Haar probability measure on the base $\Pic^d(\Sigma)$, gives rise to the probability measure $\mu_d$ on $\mathcal{M}_d(\Sigma)$.\\

Large deviations estimates  of overcrowding and undercrowding of zeros of random entire functions and random holomorphic sections been intensively studied, see \cite[Corollaire 7.4]{dinhsibony},  \cite{kris},\cite{over},   and \cite{sodin}. The main difference here is that the equation defining a zero of a holomorphic section or of an entire function is linear, whereas we will see that the one defining a critical points of a branched covering is quadratic and then  computations and estimates cannot be done purely  by Gaussian methods.\\

The paper is organized as follows. In Section \ref{ramframwork} we construct the probability measure $\mu_d$ on the space $\mathcal{M}_d(\Sigma)$ of degree $d$ branched coverings. 
In Section \ref{largesection}, we prove a large deviations estimate for the $L^1$-norm of the random variable $u\in\mathcal{M}_d(\Sigma)\mapsto |\log \norm{du}|$, see Proposition \ref{fondame}. H\"ormander peak sections and Bergman kernel estimates will play an important role in the proof of this large deviations estimate.
Finally, in Section \ref{prove} we combine these large deviations estimates together with Poincar\'e-Lelong formula to get the main theorems.
\section{Framework and probability measure on $\mathcal{M}_d(\Sigma)$}\label{ramframwork}
\subsection{Branched coverings and line bundles}
Throughout all the paper, $\Sigma$ will denote a smooth closed Riemann surface.
\begin{prop}\label{proj} Let $\mathcal{L}$ be a degree $d$ line bundle over $\Sigma$ and $(\alpha,\beta)\in H^0(\Sigma;\mathcal{L})^2$ two global sections without common zeros, then the map $u_{\alpha\beta}:\Sigma\rightarrow\C \mathbb{P}^1$ defined by $x\mapsto [\alpha(x):\beta(x)]$ is a degree $d$ branched covering. Two pairs  $(\alpha,\beta)$ and $(\alpha',\beta')$ of global holomorphic sections of $\mathcal{L}$ define the same  branched covering if and only if  $(\alpha',\beta')=(\lambda\alpha,\lambda\beta)$ for some $\lambda\in\C^*$.
\end{prop}
\begin{proof}
If $(\alpha',\beta')=(\lambda\alpha,\lambda\beta)$ for some $\lambda\in\C^*$ then it is obvious that we get the same branched covering.
Suppose now that two pairs  $(\alpha,\beta), (\alpha',\beta')$ of  holomorphic sections of $\mathcal{L}$ define the same  branched covering. In particular the sets $u_{\alpha\beta}^{-1}([0:1])$ and $u_{\alpha'\beta'}^{-1}([0:1])$ coincide. This implies that $\alpha$ and $\alpha'$ have the same zeros so that $\alpha=\lambda'\alpha'$ for some $\lambda'\in\C^*$. Taking the preimage of $[1:0]$, with the same argument we get $\beta=\lambda''\beta'$ for some $\lambda''\in\C^*$. Taking a point $x$ in the preimage of $[1:1]$ we get $\alpha(x)=\beta(x)$ and $\lambda'\alpha(x)=\lambda''\beta(x)$. This gives  $\lambda'=\lambda''$, hence the result.
\end{proof}

\begin{defn} We denote by $\mathcal{M}_d(\Sigma)$  the space of  degree $d$ branched coverings $u:\Sigma\rightarrow \C \mathbb{P}^1$ from $\Sigma$ to the Riemann sphere $\C \mathbb{P}^1$.
\end{defn}

\begin{prop}\label{fb} The space $\mathcal{M}_d(\Sigma)$ is  fibered over the space  $\Pic^d({\Sigma})$ of degree $d$ line bundles over $\Sigma$. The fibration is given by $u\in\mathcal{M}_d(\Sigma)\mapsto u^*\mathcal{O}(1)\in\Pic^d({\Sigma})$. The fiber $\mathcal{M}_d(\Sigma,{\mathcal{L}})$  over $\mathcal{L}\in\Pic^d({\Sigma})$ is the dense open subset of $\mathbb{P}(H^0(\Sigma;\mathcal{L})^2)$  given by (the class of) pair of sections $(\alpha,\beta)$ without common zeros.
\end{prop}
\begin{proof}
Given a degree $d$ branched covering $u:\Sigma\rightarrow\C \mathbb{P}^1$, we get a degree $d$ line bundle   $u^*\mathcal{O}(1)$ over $\Sigma$ and two global holomorphic sections  $u^*x_0,u^*x_1\in H^0(\Sigma;u^*\mathcal{O}(1))$ without common zeros.
Conversely, if we have a degree $d$ line bundle $\mathcal{L}\rightarrow \Sigma$ and two global sections $(\alpha,\beta)\in H^0(\Sigma;\mathcal{L})^2$ without common zeros, then the map $u_{\alpha\beta}:\Sigma\rightarrow\C \mathbb{P}^1$ defined by $x\mapsto [\alpha(x):\beta(x)]$ is a  well-defined degree $d$ branched covering. 
By Proposition \ref{proj},  $u_{\alpha\beta}=u_{\alpha'\beta'}$ if and only if  $(\alpha',\beta')=(\lambda\alpha,\lambda\beta)$ for some $\lambda\in\C^*$. Hence the result.
\end{proof}
\subsection{$L^2$-products and Bergman kernel}\label{bergman} Let $\mathcal{L}$ be a degree $d$ line bundle over $\Sigma$. In this section we construct a $L^2$-Hermitian product on $H^0(\Sigma;\mathcal{L})^2$. This Hermitian product is associated with a volume form $\omega$ of total volume $1$, which is fixed once for all.
\begin{prop}\label{metr} Let $\mathcal{L}$ be a degree $d$ line bundle over $\Sigma$ and $\omega$ a volume form on $\Sigma$ of mass $1$. Then, there exists an unique  Hermitian metric $h$ (up to multiplication by a positive constant) such that $c_1(\mathcal{L},h)=d\cdot\omega$.
\end{prop}
\begin{proof} We start with any Hermitian metric $h_0$ of  $\mathcal{L}$. Its curvature equals $\omega_0=\frac{1}{2i\pi}\partial\bar{\partial}\phi_0$, where $\phi_0=\log h_0(e_{\mathcal{L}},e_{\mathcal{L}})$ is its local potential and $e_{\mathcal{L}}$ is any non-vanishing local section. As $[\omega_0]=[d\cdot\omega]\in H^{1,1}(\Sigma;\R)$, by the $\partial\bar{\partial}$-lemma we have $d\cdot\omega=\omega_0+\frac{1}{2i\pi}\partial\bar{\partial}f$, for $f\in \mathcal{C}^{\infty}(\Sigma)$.  Then, the curvature of the Hermitian metric $h\doteqdot e^{f}h_0$  equals $\omega$. If $\tilde{h}$ is another Hermitian metric, then we have $\tilde{h}=e^{c}h$ , where $c$ is a real function $\Sigma\rightarrow \R$. If we suppose that $c_1(\mathcal{L},\tilde{h})=d\cdot\omega$, then we obtain $\partial\bar{\partial}c=0$,  which implies that $c$ is constant as $\Sigma$ is compact. Hence the result.
\end{proof}
A Hermitian metric $h$ on $\mathcal{L}$ induces a $L^2$-Hermitian product $\langle\cdot,\cdot \rangle_{L^2}$ on $H^0(\Sigma;\mathcal{L})$.  It is  defined by 
$$\langle\alpha,\beta\rangle_{L^2}=\int_{x\in\Sigma}h_x(\alpha(x),\beta(x))\omega$$
for all $\alpha,\beta$ in $H^0(\Sigma;\mathcal{L})$. The induced Hermitian product   on $H^0(\Sigma;\mathcal{L})^2$ is still denoted by $\langle\cdot,\cdot \rangle_{L^2}$.\\
Throughout all the paper, the  Hermitian metric on $\mathcal{L}$ we will consider is the one given by Proposition \ref{metr}.

Let $\mathcal{F}$ and $\mathcal{E}$ be respectively a degree $1$ and $0$ line bundles on $\Sigma$. We equip $\mathcal{F}$ and $\mathcal{E}$ by the  Hermitian metrics given by Proposition \ref{metr} which we denote by $h_{\mathcal{F}}$ and $h_{\mathcal{E}}$. In particular the metric $h_{\mathcal{F}}^d\otimes h_{\mathcal{E}}$ on $\mathcal{F}^d\otimes\mathcal{E}$ is such that its curvature equals $d\cdot\omega$. We denote by  $\mathcal{K}_{\mathcal{F}^d\otimes\mathcal{E}}(z,w)$  the Bergman kernel associated with the Hermitian line bundle $(\mathcal{F}^d\otimes\mathcal{E},h_{\mathcal{F}}^d\otimes h_{\mathcal{E}})$. The estimates of the Bergman kernel are well known, see \cite[Section 4.2]{ma2} or \cite{ber1,tian,zel}.  
In particular, along the diagonal, we have that $\mathcal{K}_{\mathcal{F}^d\otimes\mathcal{E}}(x,x)=\frac{d}{\pi}+O(1)$, $ \frac{\partial}{\partial z}\mathcal{K}_{\mathcal{F}^d\otimes\mathcal{E}}(x,x)=O(\sqrt{d})$ and $\frac{\partial^2}{\partial z\partial w}\mathcal{K}_{\mathcal{F}^d\otimes\mathcal{E}}(x,x)=\frac{d^2}{\pi}+O(d)$.
For any $x\in\Sigma$, let us consider the evaluation map $ev_x:H^0(\Sigma,\mathcal{F}^d\otimes\mathcal{E})\rightarrow(\mathcal{F}^d\otimes\mathcal{E})_x$ defined by $s\mapsto s(x)$. We denote by $\sigma_0$ the section of unit $L^2$-norm which generates the orthogonal of $\ker ev_x$. Similarly, we consider the map $j^1_x:\ker ev_x\rightarrow (\mathcal{F}^d\otimes\mathcal{E})_x\otimes T^*_{\Sigma}$ and we denote by $\sigma_1$ the section of unit $L^2$-norm generating the orthogonal of $\ker j^1_x$. We call $\sigma_0$ and $\sigma_1$ the \emph{peak sections} at $x$.

\begin{prop}\label{peak} Let $\mathcal{F}$ and $\mathcal{E}$ be respectively a degree $1$ and $0$ line bundles on $\Sigma$ and $x\in\Sigma$ be a point. Let $\sigma_0$ and $\sigma_1$ be the peak sections at $x$ associated with the line bundle $\mathcal{F}^d\otimes \mathcal{E}$. Then $\norm{\sigma_0(x)}^2\sim \frac{d}{\pi}$ and $\norm{\nabla\sigma_1(x)}^2\sim \frac{d^2}{\pi}$ as $d\rightarrow\infty$. %Moreover, if the distance between two points $x$ and $y$ us bigger than $\frac{\log d}{\sqrt{d}}$, then $\sigma_i^y$ and $\sigma^x_j$, $i,j\in\{0,1\}$, are asymptotically orthogonal. More precisely we have $\langle \sigma^x_j,\sigma_i^y\rangle_{L^2}=O(\frac{1}{d})$. Moreover, the pointwise norm of $\sigma_i^x(y)$ and $\nabla\sigma_i^x(y)$, $i,j\in\{0,1\}$, is a $O(\frac{1}{d})$.
\end{prop}
\begin{proof} 
%This is a classical fact (see for example \ref{tian}). We will sketch the proof for the sake of completeness.
We complete the orthonormal family $\{\sigma_0,\sigma_1\}$ into an orthonormal basis $\{\sigma_0,\sigma_1,\sigma_2,\dots,\sigma_{N_d-1}\}$ of $H^0(\Sigma;\mathcal{F}^d\otimes \mathcal{E})$. Then, $\sum_{i=0}^{N_d-1}\norm{\sigma_i(x)}^2$ equals the Bergman kernel $\mathcal{K}_{\mathcal{F}^d\otimes \mathcal{E}}(x,x)$. Similarly, we have $\sum_{i=0}^{N_d-1}\norm{\nabla\sigma_i(x)}^2=\frac{\partial^2}{\partial z\partial w}\mathcal{K}_{\mathcal{\mathcal{F}}^d\otimes \mathcal{E}}(x,x)$. Now, we know that, as $d$ goes to infinity, $\mathcal{K}_{\mathcal{F}^d\otimes \mathcal{E}}(x,x)=\frac{d}{\pi}+O(1)$ and $\frac{\partial^2}{\partial z\partial w}\mathcal{K}_{\mathcal{F}^d\otimes \mathcal{E}}(x,x)=\frac{d^2}{\pi}+O(d)$ and that, by the construction of the peak sections $\sigma_0$ and $\sigma_1$, we have $\sum_{i=0}^{N_d-1}\norm{\sigma_i(x)}^2=\norm{\sigma_0(x)}^2$ and $\sum_{i=0}^{N_d-1}\norm{\nabla\sigma_i(x)}^2=\norm{\nabla\sigma_0(x)}^2+\norm{\nabla\sigma_1(x)}^2$. Now, it is easy to see that $\norm{\nabla\sigma_0(x)}^2=\frac{\frac{\partial}{\partial z}\mathcal{K}_{\mathcal{F}^d\otimes \mathcal{E}}(x,x)}{\sqrt{\mathcal{K}_{\mathcal{F}^d\otimes \mathcal{E}}(x,x)}}$ and the latter is a $O(1)$. Hence the result.
\end{proof}

\subsection{Probability on $\mathcal{M}_d(\Sigma)$}\label{sectproba}
Let $\Sigma$ be a closed Riemann surface equipped with a volume form $\omega$ of total mass $1$.
In this section, we construct  a natural probability measure on the space $\mathcal{M}_d(\Sigma)$ of degree $d$ branched coverings from $\Sigma$ to $\mathbb{C}\mathbb{P}^1$.\\
Given $\mathcal{L}$ a degree $d$ line bundle over $\Sigma$, we have seen that a  Hermitian metric $h$  induces a natural $L^2$-Hermitian product on $H^0(\Sigma;\mathcal{L})$ and then  on $H^0(\Sigma;\mathcal{L})^2$.
%Remark that, by Proposition \ref{proj}, it is  natural to projectivize $H^0(\Sigma;\mathcal{L})^2$  if we want to study  branched coverings. 
The $L^2$-product  on $H^0(\Sigma;\mathcal{L})^2$ induces a Fubini-Study metric on  $\mathbb{P}(H^0(\Sigma;\mathcal{L})^2)$.
We recall that the Fubini-Study metric is constructed as follows. First we restrict the Hermitian product to the unit sphere of $H^0(\Sigma;\mathcal{L})^2$. The obtained metric is then invariant under the action of $S^1=U(1)$. The Fubini-Study metric is then the quotient metric on $\mathbb{P}(H^0(\Sigma;\mathcal{L})^2)$. 
\begin{defn}\label{measurel} Let $\mathcal{L}$ be a  line bundle over $\Sigma$. We denote by  $\mu_{\mathcal{L}}$ the probability measure on  $\mathbb{P}( H^0(\Sigma;\mathcal{L})^2)$ induced by the normalized Fubini-Study volume form. Here, the Fubini-Study metric on $\mathbb{P}(H^0(\Sigma;\mathcal{L})^2)$ is the one induced by the Hermitian metric on $\mathcal{L}$ given by Proposition \ref{metr}.
\end{defn}
\begin{prop}\label{fs} The probability measure $\mu_{\mathcal{L}}$ over  $\mathbb{P}( H^0(\Sigma;\mathcal{L})^2)$ does not depend on the choice of the multiplicative constant in front  of the metric $h$ given by Proposition \ref{metr}.
\end{prop}
\begin{proof}
Fix $h$ a metric given by Proposition \ref{metr}. If we multiply this metric by a positive constant $e^c$, then the two $L^2$-scalar products $\langle\cdot,\cdot\rangle$ and $\langle\cdot,\cdot\rangle_c$  induced respectively by $h$ and $e^{c}h$  are equal up to a multiplication by a positive scalar, that is $\langle\cdot,\cdot\rangle_c=e^{c}\langle\cdot,\cdot\rangle$. This constant in front of the scalar product does not affect the Fubini-Study metric, once we renormalize the Fubini-Study volume to have total volume $1$.
\end{proof}

\begin{prop}[Proposition 2.11 of \cite{anc}] Let $\mathcal{L}$ be a degree $d$ line bundle over $\Sigma$. For almost all $[\alpha,\beta]\in\mathbb{P}( H^0(\Sigma;\mathcal{L})^2)$, the map $u_{\alpha\beta}:x\in\Sigma\mapsto [\alpha(x):\beta(x)]\in\C\mathbb{P}^1$ is a degree $d$ branched covering.
\end{prop}
\begin{defn} Let  $\Pic^d(\Sigma)$ be the space of degree $d$ line bundles over $\Sigma$. It is a principal space under the action of  $\Pic^0(\Sigma)$ (by tensor product) and so it inherits a normalized Haar measure that we denote by  $\dH$.
\end{defn}
Recall that we denote by $\mathcal{M}_d(\Sigma,\mathcal{L})$ the fiber of the map $\mathcal{M}_d(\Sigma)\rightarrow \Pic^d(\Sigma)$ given by Proposition \ref{fb}.
 We will denote by $\Lambda_{\mathcal{L}}$ the set of pair of sections of $H^0(\Sigma;\mathcal{L})^2$ having at least one common zero, so that $\mathcal{M}_d(\Sigma,{\mathcal{L}})=\mathbb{P}(H^0(\Sigma;\mathcal{L})^2\setminus\Lambda_{\mathcal{L}}).$
\begin{defn}\label{hfs} We define the  probability measure $\mu_d$  on $\mathcal{M}_d(\Sigma)$ by
$$\int_{\mathcal{M}_d(\Sigma)}f\textrm{d}\mu_d=\int_{\mathcal{L}\in\Pic^d(\Sigma)}(\int_{\mathcal{M}_d(\Sigma,\mathcal{L})}f\textrm{d}\mu_{\mathcal{L}})\dH(\mathcal{L})$$
for any $f\in \mathcal{M}_d(\Sigma)$ measurable function. Here: 
\begin{itemize}
\item  $\mu_{\mathcal{L}}$ denotes (by a slight abuse of notation)  the restriction to $\mathcal{M}_d(\Sigma,\mathcal{L})$ of the probability measure on $\mathbb{P}(H^0(\Sigma,\mathcal{L})^2)$ defined in Definition \ref{measurel}.  
\item $\dH$ denotes the normalized Haar measure on $\Pic^d(\Sigma)$.
\end{itemize}
%For the probability measure $d\eta^{\C}$ on $\mathcal{M}_d(\Sigma)$ the definition is equivalent.
\end{defn}
\begin{oss} The choice  the Haar measure on $\Pic^d(\Sigma)$ is natural but not essential: all the results of this paper are still true if we choose any probability measure which is absolute continuous with respect to the Haar measure. In the study  of  complex zeros of random holomorphic sections of a line bundle over a Riemann surface,   a similar construction was given by Zelditch in \cite{zeldlarge}.
\end{oss}

\subsection{Gaussian vs Fubini-Study measure}\label{gaus}
Following \cite{sz}, given  $\mathcal{L}\in\Pic^d(\Sigma)$  a degree $d$  line bundle,  we equip $ H^0(\Sigma;\mathcal{L})^2$ with  a Gaussian measure  $\gamma_{\mathcal{L}}$. In order to do this, 
we fix a  volume form $\omega$ of total volume $1$ on $\Sigma$ and we equip $\mathcal{L}$ by the metric $h$ with curvature $d\cdot\omega$ (the metric $h$ is unique up to a multiplicative constant, see Proposition \ref{metr}).\\
We have seen that any Hermitian metric $h$ induces  a $L^2$-Hermitian product on the space $H^0(\Sigma;\mathcal{L})$ of global holomorphic sections of $\mathcal{L}$ denoted by $\langle\cdot,\cdot \rangle_{L^2}$ and  defined by 
$$\langle\alpha,\beta\rangle_{L^2}=\int_{x\in\Sigma}h_x(\alpha(x),\beta(x))\omega$$
for all $\alpha,\beta$ in $H^0(\Sigma;\mathcal{L})$. The Gaussian measure $\gamma_{\mathcal{L}}$ on  $H^0(\Sigma;\mathcal{L})^2$ is defined  by
$$\gamma_{\mathcal{L}}(A)=\frac{1}{\pi^{2N_d}}\int_{(\alpha,\beta)\in A}e^{-\norm{\alpha}_{L^2}^2-\norm{\beta}_{L^2}^2}\textrm{d}\alpha \textrm{d}\beta$$
for any open subset  $A\subset H^0(\Sigma;\mathcal{L})^2$. Here  $\textrm{d}\alpha\textrm{d}\beta$ is the Lebesgue measures on $( H^0(\Sigma;\mathcal{L})^2;\langle\cdot,\cdot\rangle_{L^2})$ and $N_d$ denotes the complex dimension of $H^0(\Sigma;\mathcal{L})$. If $d>2g-2$, where $g$ is the genus of $\Sigma$, then $H^1(\Sigma;\mathcal{L})=0$ and then, by Riemann-Roch theorem, we have  $N_d=d+1-g$.
\begin{prop}\label{vs} Let $f$ be a function on a Hermitian space $(V,\langle\cdot,\cdot\rangle)$ which is constant over the complex lines, i.e. $f(v)=f(\lambda v)$ for any $v\in V$ and any $\lambda\in\C^*$. Denote  by $\gamma$  the Gaussian measure on $V$ induced by $\langle\cdot,\cdot\rangle$ and by $\mu$  the normalized Fubini-Study measure on the projectivized $\mathbb{P}(V)$.
Then,  we have 
$$\int_{V}f\textrm{d}\gamma=\int_{\mathbb{P}(V)}[f]\textrm{d}\mu$$
where $[f]$ is the function on $\mathbb{P}(V)$ induced by $f$. 
\end{prop}
\begin{proof}
 This is a direct consequence of the construction of the Fubini-Study metric.
\end{proof}

 The fundamental consequence of  Proposition \ref{vs} is that, if we want to integrate a function on  $\mathbb{P}(H^0(\Sigma;\mathcal{L})^2)$,  we could pull-back this function over $H^0(\Sigma;\mathcal{L})^2$  and integrate this pull-back with respect to the Gaussian measure induced by  \emph{any} Hermitian metric on $\mathcal{L}$ given by Proposition \ref{metr}. %(that is the value of the integral does not depend on the overall multiplicative constant of Proposition \ref{metr}).
  In particular, for any $A\subset \mathbb{P}(H^0(\Sigma;\mathcal{L})^2)$, we have $\mu_{\mathcal{L}}(A)=\gamma_{\mathcal{L}}(\pi^{-1}(A))$ where $\pi:H^0(\Sigma,\mathcal{L})^2\rightarrow\mathbb{P}(H^0(\Sigma,\mathcal{L})^2)$ is the natural projection.\\
   %The Gaussian model $( H^0(\Sigma;\mathcal{L})^2,\gamma_{\mathcal{L}})$ and the Fubini-Study one $(\mathbb{P}( H^0(\Sigma;\mathcal{L})^2),\mu_{\mathcal{L}})$ are  then equivalent for our purpose.
%The reason why we introduce this model is that the computations are more pleasant  in the Gaussian model rather than in the Fubini-Study one.\\

\section{Critical points and large deviations estimates}\label{largesection}
\subsection{Wronskian and critical points}
Let $\Sigma$ be a closed Riemann surface and $\mathcal{L}$ be a degree $d$ line bundle over $\Sigma$. In this section, we start the study of the critical points of a branched covering by seeing them as zeros of a global section, the Wronskian. 
\begin{defn}\label{wr} Let $\nabla$ be a connection on $\mathcal{L}$. For any pair of sections $(\alpha,\beta)\in H^0(\Sigma,\mathcal{L})^2$, we denote by $W_{\alpha\beta}$ the Wronskian $\alpha\otimes\nabla\beta-\beta\otimes\nabla\alpha$, which is a global section of $\mathcal{L}^{2}\otimes T^*_{\Sigma}$.
\end{defn}
\begin{oss}   The Wronskian $W_{\alpha\beta}$ does not depend on the choice of a connection on $\mathcal{L}$. Indeed, two connections $\nabla$ and $\nabla'$ on $\mathcal{L}$ differ by a $1$-form $\theta$, and then $\alpha\otimes(\nabla-\nabla')\beta-\beta\otimes(\nabla-\nabla')\alpha=\alpha\otimes\beta\otimes\theta-\beta\otimes\alpha\otimes\theta=0$.
\end{oss}
\begin{prop}\label{bonnedef} Let $\mathcal{L}$ be a degree $d$ line bundle over $\Sigma$ and $(\alpha,\beta)\in H^0(\Sigma,\mathcal{L})^2$ be a pair of sections without common zeros. A point $x\in\Sigma$ is a critical point of the map $u_{\alpha\beta}:x\in\Sigma\mapsto [\alpha(x):\beta(x)]\in\C\mathbb{P}^1$ is and only if it is a zero of the Wronskian $W_{\alpha\beta}$ defined in Definition \ref{wr}.
\end{prop}
\begin{proof}
Let $(\alpha,\beta)\in H^0(\Sigma,\mathcal{L})^2$ be a pair of sections without common zeros and   $x\in\Sigma$.  Suppose that $\beta(x)\neq 0$. On $\C\mathbb{P}^1$ we consider the coordinate chart $z=\frac{z_1}{z_0}$, where $[z_0,z_1]$ are the standard  homogeneous coordinates of $\C\mathbb{P}^1$. Under this chart, the branched covering $u_{\alpha\beta}$ equals the meromorphic function $\frac{\alpha}{\beta}$. The differential of $\frac{\alpha}{\beta}$  equals $\frac{\alpha\otimes\nabla\beta-\beta\otimes\nabla\alpha}{\beta^2}=\frac{W_{\alpha\beta}}{\beta^2}$ and then, as $\beta(x)\neq 0$, we get that $x$ is a critical point of $u_{\alpha\beta}$ if and only if $W_{\alpha\beta}(x)=0$. If we suppose that $\alpha(x)\neq 0$, we use the coordinate $w=\frac{z_0}{z_1}$ and   the same computation as before  gives us that $u_{\alpha\beta}$ equals the function $\frac{\beta}{\alpha}$ whose differential is  $\frac{-W_{\alpha\beta}}{\alpha^2}$. Hence the result.
\end{proof}
\begin{defn}\begin{itemize}
\item For any  branched covering $u\in\mathcal{M}_d(\Sigma)$ we denote by $T_u$ the probability empirical measure   associated with the critical points of $u$, that is $$T_u=\frac{1}{2d+2g-2}\displaystyle\sum_{x\in\Crit(u)}\delta_x.$$
Here, $\delta_x$ is the Dirac measure at $x$.
\item For any pair  $(\alpha,\beta)$ of  global sections of $\mathcal{L}$, the empirical probability measure on the  critical points of $u_{\alpha\beta}$ is simply denoted by $T_{\alpha\beta}$ (instead of $T_{u_{\alpha\beta}}$). 
\end{itemize}
\end{defn}
 \begin{thm}\label{expected} Let $(\Sigma,\omega)$ be a closed Riemann surface equipped with a volume form of total volume equal to $1$. Then $$\lim_{d\rightarrow\infty}\mathbb{E}[T_u]\rightarrow \omega$$ 
weakly in the sense of distribution. Here, the expected value is taken with respect the probability measure $\mu_d$ on $\mathcal{M}_d(\Sigma)$ defined in Definition \ref{hfs}.
\end{thm}
\begin{proof} We fix a degree $1$ line bundle $\mathcal{F}$ over $\Sigma$, so that for any $\mathcal{L}\in\Pic^d(\Sigma)$ there exists an unique $\mathcal{E}\in \Pic_0(\Sigma)$ such that $\mathcal{L}^d\otimes \mathcal{E}$. We equip $\mathcal{F}$ and $\mathcal{E}$ by the  Hermitian metrics given by Proposition \ref{metr}. We denote this metric respectively by $h_{\mathcal{F}}$ and $h_{\mathcal{E}}$. Then, by \cite[Theorem 1.5]{anc} and by Proposition \ref{vs}, we have $$\lim_{d\rightarrow\infty}\mathbb{E}_{\mathcal{F}^d\otimes \mathcal{E}}[T_{\alpha\beta}]\rightarrow \omega$$ 
weakly in the sense of distribution. Here, $\mathbb{E}_{\mathcal{F}^d\otimes \mathcal{E}}$ stands for the expected value of  with respect to the probability measure $\mu_{\mathcal{F}^d\otimes \mathcal{E}}$ defined in Definition \ref{measurel}. 
The result follows by integrating  along the compact base $\Pic^0(\Sigma)\simeq \Pic^d(\Sigma)$.
\end{proof}
\subsection{Large deviations estimates}
Let $\Sigma$ be a closed Riemann surface equipped with a volume form $\omega$ of total mass $1$ and $\mathcal{F}$ and $\mathcal{E}$ be respectively a degree $1$ and $0$ line bundle over $\Sigma$.  
We fix the Hermitian metrics  $h_{\mathcal{F}}$ and $h_{\mathcal{E}}$ on $\mathcal{F}$ and $\mathcal{E}$  given by Proposition \ref{metr} and we will denote by $\norm{\cdot}$ any norm induced by these Hermitian metrics, in particular the Hermitian metric induced on $\mathcal{F}^{2d}\otimes\mathcal{E}^2\otimes T^*_{\Sigma}$.
Let $\omega_d$ be the curvature form of $\mathcal{F}^{2d}\otimes\mathcal{E}^2\otimes T^*_{\Sigma}$, which equals   $\omega_d=2d\cdot\omega+O(1)$.
Recall that we denote by $W_{\alpha\beta}=\alpha\otimes\nabla\beta-\beta\otimes\nabla\alpha$ the Wronskian of a pair $(\alpha,\beta)$ of global sections of $\mathcal{F}^d\otimes\mathcal{E}$, see Definition \ref{wr}.
The goal of this section is to prove Proposition \ref{fondame}, which is a large deviations estimate for the $L^1$-norm of $\log \norm{W_{\alpha\beta}(x)}$.   This is a key result for the proof of Theorem \ref{smooth}.
\begin{prop}\label{fondame} Let $\epsilon_d$ be a sequence of positive numbers of the form $\epsilon_d=O(d^{-a})$, for some $a\in[0,1)$. Then, there exists a positive constant $C$ such that 
$$\gamma_{\mathcal{F}^d\otimes \mathcal{E}}\big\{(\alpha,\beta)\in H^0(\Sigma,\mathcal{F}^d\otimes\mathcal{E}), \int_{\Sigma}\big|\log \norm{W_{\alpha\beta}(x)}\big|\geq \epsilon_d d\big\}\leq \exp(-C\epsilon_d d).$$
Here, $\gamma_{\mathcal{F}^d\otimes \mathcal{E}}$ is the Gaussian measure on $H^0(\Sigma,\mathcal{F}^d\otimes \mathcal{E})^2$  constructed in Section \ref{gaus}.
\end{prop}
In order to prove Proposition \ref{fondame}, we need some results on large deviations estimates on the modulus of $\log \norm{W_{\alpha\beta}(x)}$. 
For this purpose, we will use Bergman kernel estimates as well as peak sections associated with the Hermitian line bundle $(\mathcal{F}^d\otimes\mathcal{E},h^d_{\mathcal{F}}\otimes h_{\mathcal{E}})$ of positive curvature $d\cdot \omega$.

 \begin{prop}\label{above} For any sequence $\epsilon_d$ of positive real numbers, we have 
$$\gamma_{\mathcal{F}^d\otimes \mathcal{E}}\big\{(\alpha,\beta), \sup_{x\in \Sigma}\norm{W_{\alpha\beta}(x)}\geq e^{\epsilon_d d}\big\}\leq 4d^2\exp(-\frac{e^{\frac{\epsilon_d}{2}d}}{2}).$$
Here, $\gamma_{\mathcal{F}^d\otimes \mathcal{E}}$ is the Gaussian measure on $H^0(\Sigma,\mathcal{F}^d\otimes \mathcal{E})^2$  constructed in Section \ref{gaus}.
\end{prop}
\begin{proof}
Let $N_d$ be the dimension of $H^0(\Sigma,\mathcal{F}^d\otimes \mathcal{E})$. By Riemann-Roch theorem, we have that $N_d=d+1-g=d+O(1)$ as $d>2g-2$, where $g$ is the genus of $\Sigma$.
 Let $s_1,\dots,s_{N_d}$ be an orthonormal basis of $H^0(\Sigma,\mathcal{F}^d\otimes \mathcal{E})$ and write $\alpha=\displaystyle\sum_{i=1}^{N_d}a_is_i$ and  $\beta=\displaystyle\sum_{i=1}^{N_d}b_is_i$, for any $\alpha,\beta\in H^0(\Sigma;\mathcal{F}^d\otimes\mathcal{E})$. Now, $\sup_{x\in \Sigma}\norm{\alpha\otimes\nabla\beta-\beta\otimes\nabla\alpha}>e^{\epsilon_d d}$ if and only if  $\sup_{x\in \Sigma}\norm{\displaystyle\sum_{i,j}(a_ib_j-a_jb_i)s_i\otimes\nabla s_j}>e^{\epsilon_d d}$. Now, using first the triangular inequality and then Cauchy-Schwarz, we have 
 \begin{multline}\label{suite}
 \norm{\displaystyle\sum_{i,j}(a_ib_j-a_jb_i)s_i\otimes\nabla s_j}^2\leq \big(\displaystyle\sum_{i,j}|a_ib_j-a_jb_i|\cdot\norm{s_i\otimes\nabla s_j}\big)^2\leq 
 \displaystyle\sum_{i,j}|a_ib_j-a_jb_i|^2\displaystyle\sum_{i,j}\norm{s_i\otimes\nabla s_j}^2 \\
  \leq \displaystyle\sum_{i,j}|a_ib_j-a_jb_i|^2\sqrt{\sum_{i=1}^{N_d}\norm{s_i}^2}\cdot\sqrt{\displaystyle\sum_{j=1}^{N_d}\norm{\nabla s_j}^2} 
 \end{multline}  
  By Bergman kernel estimates (see \cite{ber1,zel,tian}), we have for any $x\in\Sigma$ $$\sqrt{\displaystyle\sum_{i=1}^{N_d}\norm{s_i(x)}^2}\sim \frac{\sqrt{d}}{\sqrt{\pi}} \hspace{2mm} \textrm{and} \hspace{2mm}  \sqrt{\displaystyle\sum_{j=1}^{N_d}\norm{\nabla s_j(x)}^2}\sim \frac{d}{\sqrt{\pi}}$$ so that the last expression in (\ref{suite}) is bigger than $e^{2\epsilon_d d}$ if    $\displaystyle\sum_{i,j}|a_ib_j-a_jb_i|^2>\pi e^{2\epsilon_d d}d^{-\frac{3}{2}}$ and this holds if $N_d\cdot\max_{i,j}|a_ib_j-a_jb_i|^2>\pi e^{2\epsilon_d d}d^{-\frac{3}{2}}$.
 We then have  $$\big\{\sup_{x\in \Sigma}\norm{(\alpha\otimes\nabla\beta-\beta\otimes\nabla\alpha)(x)}\geq e^{\epsilon d}\big\}\subseteq \big\{\max_{i,j}|a_ib_j-a_jb_i|^2>\pi e^{2\epsilon_d d}d^{-\frac{3}{2}}N_d^{-1}\big\}\subseteq \big\{\max_{i,j}|a_ib_j-a_jb_i|>e^{\epsilon_d d}\big\}.$$
 We then have that  
 \begin{equation}\label{confronto}
 \gamma_{\mathcal{F}^d\otimes \mathcal{E}}\big\{(\alpha,\beta), \sup_{x\in \Sigma}\norm{W_{\alpha\beta}(x)}\geq e^{\epsilon_d d}\big\}\leq\gamma_{\mathcal{L}}\big\{\max_{i,j}|a_ib_j-a_jb_i|>e^{\epsilon_d d}\big\}\leq d^2\cdot\gamma_{\mathcal{F}^d\otimes \mathcal{E}}\big\{|a_ib_j-a_jb_i|>e^{\epsilon_d d}\big\}.
  \end{equation}
 We then estimate the last measure in (\ref{confronto}). We write $a=(a_i,a_j)\in\C^2$ and $b=(b_i,b_j)$, for any $i,j\in\{1,\dots, N_d\}$.
 By Cauchy-Schwarz we have $(a_ib_j-a_jb_i)^2\leq (|a_i|^2+|a_j|^2)(|b_i|^2+|b_j|^2)$ so that 
 $$\gamma_{\mathcal{L}}\big\{|a_ib_j-a_jb_i|>e^{\epsilon_d d}\big\}\leq\frac{1}{\pi^4}\int_{|a|\cdot|b|> e^{\epsilon_d d}}e^{-|a|^2-|b|^2}\textrm{d}a\textrm{d}\bar{a}\textrm{d}b\textrm{d}\bar{b}=\frac{2}{\pi^4}\int_{\substack{|a|\cdot|b|> e^{\epsilon_d d} \\ |a|>|b|}}e^{-|a|^2-|b|^2}\textrm{d}a\textrm{d}\bar{a}\textrm{d}b\textrm{d}\bar{b}$$
$$ \leq \frac{2}{\pi^4}\int_{|a|> e^{\frac{\epsilon_d}{2} d}}\int_{b\in \C^2}e^{-|a|^2-|b|^2}\textrm{d}a\textrm{d}\bar{a}\textrm{d}b\textrm{d}\bar{b}=\frac{1}{\pi^2}\int_{|a|> e^{\frac{\epsilon_d}{2} d}}e^{-|a|^2}\textrm{d}a\textrm{d}\bar{a}$$
$$\leq \frac{1}{\pi^2}e^{\frac{-e^{\frac{\epsilon_d}{2} d}}{2}}\int_{a\in\C^2}e^{-\frac{|a|^2}{2}}\textrm{d}a\textrm{d}\bar{a}=4e^{\frac{-e^{\frac{\epsilon_d}{2} d}}{2}}.$$
Combining the last estimate with (\ref{confronto}) we have the result.
\end{proof}

\begin{prop}\label{below} For any sequences $\epsilon_d$ of positive real numbers  and any $x\in\Sigma$, we have 
$$\gamma_{\mathcal{F}^d\otimes \mathcal{E}}\big\{(\alpha,\beta), \norm{W_{\alpha\beta}(x)}\leq e^{-\epsilon_d d}\big\}\leq \exp(-\frac{\epsilon_d}{2} d).$$
Here, $\gamma_{\mathcal{F}^d\otimes \mathcal{E}}$ is the Gaussian measure on $H^0(\Sigma,\mathcal{F}^d\otimes \mathcal{E})^2$  constructed in Section \ref{gaus}.
\end{prop}
\begin{proof}
Let $\sigma_0$ and $\sigma_1$ be the first two peak sections at $x$, as in Proposition \ref{peak}. Recall that we have the estimates  $\norm{\sigma_0(x)}\sim \frac{\sqrt{d}}{\sqrt{\pi}}$ and $\norm{\nabla\sigma_1(x)}\sim \frac{d}{\sqrt{\pi}}$. We write $\alpha=a_0\sigma_0+a_1\sigma_1 +\tau$ and $\beta=b_0\sigma_0+b_1\sigma_1 +\tau'$ where $\tau,\tau'\in\ker J^1_x$, that is $\tau(x)=\tau'(x)=0$ and $\nabla\tau(x)=\nabla\tau'(x)=0$. In particular, we have $W_{\alpha\beta}(x)\doteqdot(\alpha\otimes\nabla\beta-\beta\otimes\nabla\alpha)(x)=(a_0b_1-b_0a_1) \sigma_0(x)\otimes\nabla\sigma_1(x)$. We then have the following inclusion
$$\big\{(\alpha,\beta), \norm{W_{\alpha\beta}(x)}\leq e^{-\epsilon_d d}\big\}\subseteq \big\{(a_0\sigma_0+a_1\sigma_1 +\tau,b_0\sigma_0+b_1\sigma_1 +\tau'), |a_0b_1-b_0a_1|\leq e^{-\frac{\epsilon_d}{2} d}\big\}.$$
Now, the Gaussian measure of the last set equals
 \begin{equation}\label{measurega}
 \gamma\big\{|a_0b_1-b_0a_1|\leq e^{-\frac{\epsilon_d}{2} d}\big\}=\frac{1}{\pi^4}\int_{|a_0b_1-b_0a_1|\leq e^{-{\epsilon_d}{2} d}} e^{-|a_0|^2-|a_1|^2-|b_0|^2-|b_1|^2}\textrm{d}a\textrm{d}\bar{a}\textrm{d}b\textrm{d}\bar{b}.
 \end{equation}
 For any $a=(a_0,a_1)$ we make an unitary trasformation of $\C^2$ (of coordinates $b_0,b_1$) by sending the vector $(1,0)$ to $v_a=\frac{1}{\sqrt{|a_0|^2+|a_1|^2}}(a_0,a_1)$ and the vector $(0,1)$ to $w_a=\frac{1}{\sqrt{|a_0|^2+|a_1|^2}}(-\bar{a}_1,\bar{a}_0).$ We will write any vector of $\C^2$ as a sum $tv_a+sw_a $ with $s,t\in\C$. In these coordinates, the condition $\{(b_0,b_1)\in\C^2, |a_0b_1-b_0a_1|\leq e^{-\frac{\epsilon_d}{2} d}\}$ reads $\{(s,t)\in \C^2, |s|\cdot\norm{a}\leq e^{-\frac{\epsilon_d}{2} d}\}$.  The measure appearing in Equation (\ref{measurega}) is then equal to 
 $$\int_{a\in\C^2}\big(\int_{\substack{(t,s)\in\C^2 \\ |s|\cdot\norm{a}<e^{-\frac{\epsilon_d}{2} d}}}\frac{e^{-|t|^2-|s|^2}}{\pi^2}\textrm{d}t\textrm{d}\bar{t}\textrm{d}s\textrm{d}\bar{s}\big)\textrm{d}\gamma(a)=\int_{a\in\C^2}\big(\int_{\substack{s\in\C \\ |s|\cdot\norm{a}<e^{-\frac{\epsilon_d}{2} d}}}\frac{e^{-|s|^2}}{\pi }\textrm{d}s\textrm{d}\bar{s}\big)\textrm{d}\gamma(a).$$
 where $\textrm{d}\gamma(a)=\frac{1}{\pi^2}e^{-|a|^2}\textrm{d}a\textrm{d}\bar{a}$. \\
 We pass to polar coordinates $a=(r,\theta)$ with $\theta\in S^3$ and $s=\rho e^{i\varphi}$ with $\varphi\in S^1$. We then have $$\frac{2\Vol(S^3)}{\sqrt{\pi}^{5}}\int_{r=0}^{\infty}\big(\int_{\substack{\rho\in\R_+ \\ \rho\cdot r<e^{-\frac{\epsilon}{2} d}}}e^{-\rho^2}\rho \textrm{d}\rho\big)r^3e^{-r^2}\textrm{d}r=$$
 \begin{equation}\label{sommebelow}
 \frac{2\Vol(S^3)}{\sqrt{\pi}^{5}}\int_{r=e^{-\frac{\epsilon_d}{4}d}}^{\infty}\big(\int_{\rho=0}^{\frac{e^{-\frac{\epsilon_d}{2} d}}{r}}e^{-\rho^2}\rho \textrm{d}\rho\big)r^3e^{-r^2}\textrm{d}r+\frac{2\Vol(S^3)}{\sqrt{\pi}^{5}}\int_{r=0}^{e^{-\frac{\epsilon_d}{4}d}}\big(\int_{\rho=0}^{\frac{e^{-\frac{\epsilon_d}{2} d}}{r}}e^{-\rho^2}\rho \textrm{d}\rho\big)r^3e^{-r^2}\textrm{d}r
 \end{equation}
 The first term of the sum in (\ref{sommebelow}) is smaller than
 $$\frac{2\Vol(S^3)}{\sqrt{\pi}^{5}}\int_{r=0}^{\infty}\big(\int_{\rho=0}^{e^{-\frac{\epsilon_d}{4}d}}e^{-\rho^2}\rho \textrm{d}\rho\big)r^3e^{-r^2}\textrm{d}r=\frac{2}{\sqrt{\pi}}\int_{\rho=0}^{e^{-\frac{\epsilon_d}{4}d}}e^{-\rho^2}\rho \textrm{d}\rho$$
 $$<\frac{2}{\sqrt{\pi}}\int_{\rho=0}^{e^{-\frac{\epsilon_d}{4}d}}\rho \textrm{d}\rho=\frac{1}{\sqrt{\pi}}e^{-\frac{\epsilon_d}{2} d}$$
 The second term of the sum in (\ref{sommebelow}) is smaller than
 $$\frac{2\Vol(S^3)}{\sqrt{\pi}^{5}}\int_{r=0}^{e^{-\frac{\epsilon_d}{4}d}}\big(\int_{\rho=0}^{\infty}e^{-\rho^2}\rho \textrm{d}\rho\big)r^3e^{-r^2}\textrm{d}r=\frac{\Vol(S^3)}{\pi^4}\int_{r=0}^{e^{-\frac{\epsilon_d}{4}d}}r^3e^{-r^2}\textrm{d}r$$
 $$<\frac{\Vol(S^3)}{\pi^4}\int_{r=0}^{e^{-\frac{\epsilon_d}{4}d}}r^3\textrm{d}r=\frac{\Vol(S^3)}{4\pi^4}e^{-\epsilon_d d}.$$
 We then obtain that the measure (\ref{measurega}) is smaller than $\frac{1}{\sqrt{\pi}}e^{-\frac{\epsilon_d}{2} d}+\frac{\Vol(S^3)}{4\pi^4}e^{-\epsilon_d d}<e^{-\frac{\epsilon_d}{2} d}$.
\end{proof}

\begin{proof}[Proof of Proposition \ref{fondame}] 
Let us fix some notations. For any $r>0$, we denote the circle of radiur $r$ by $S(r)=\{|z|=r\}$  and the ball of radius $r$ by $B(r)$. 
Finally, we denote by $\log^+t=\max(\log t, 0)$ and $\log^-t=\max(-\log t, 0)$ so that $\log t=\log^+t-\log^-t$ and $|\log t|=\log^+t+\log^-t$.

By Proposition \ref{above}, we get 
\begin{equation}\label{piu}
\gamma_{\mathcal{F}^d\otimes \mathcal{E}}\big\{(\alpha,\beta), \int_{\Sigma}\log^+ \norm{W_{\alpha\beta}(x)} \geq \epsilon_d d\big\}\leq \exp(-C\epsilon_dd)
\end{equation}
so that we have to prove the following bound
\begin{equation}\label{meno}
\gamma_{\mathcal{F}^d\otimes \mathcal{E}}\big\{(\alpha,\beta), \int_{\Sigma}\log^- \norm{W_{\alpha\beta}(x)}\geq \epsilon_d d\big\}\leq \exp(-C\epsilon_d d).
\end{equation}
In order to prove (\ref{meno}), let us cover $\Sigma$ by a finite number of annuli $A_1,\dots,A_n$, each of which  is included in a coordinate chart. We can suppose that each annulus, read in these coordinates, is of the form  $B(3)\setminus B(1)$. We fix a
 holomorphic trivializations $e_{\mathcal{F}}$ and $e_{\mathcal{E}}$ of $\mathcal{F}$ and $\mathcal{E}$ over each coordinate chart and then over each annulus.
We make the following:\\ 

\textbf{Claim:} \emph{ For any sequence $\epsilon_d$ of positive real numbers, there exists a positive constant $C$ and a measurable set $E$ with $\gamma_{\mathcal{F}^d\otimes \mathcal{E}}(E)<e^{-C{\epsilon_d}d}$, such that \begin{equation}\label{key} \int_{S(r)} \log^- \norm{W_{\alpha\beta}} \textrm{d}\sigma_r\leq \epsilon_d d
\end{equation} 
for $(\alpha,\beta)\in H^0(\Sigma,\mathcal{F}^d\otimes \mathcal{E})^2\setminus E$, $r\in [1,3]$, $d\gg 0$.
Here, all the computations are done in the coordinate chart   and  $\sigma_r$ is the invariant probability measure on the circle $S(r)=\{ |z|=r\}$.}\\

Before proving the Claim, we end the proof of Proposition \ref{fondame}.
Since the exceptional set $E$ is independent of the radius $r\in [1,3]$,  we can integrate the inequality (\ref{key}) over $r\in [1,3]$ and we get  
\begin{equation}\label{inega}
\int_{B(3)\setminus B(1)} \log^- \norm{W_{\alpha\beta}} \textrm{d}\sigma_r \textrm{d}r\leq M\epsilon_d d
\end{equation} for some $M>0$ (independent of $\epsilon_d$) and any $(\alpha,\beta)\in H^0(\Sigma,\mathcal{F}^d\otimes \mathcal{E})^2\setminus E$.  By summing over the annuli the inequality (\ref{inega}) we get (\ref{meno}) which, together with (\ref{piu}), concludes the proof of the proposition.\\

 We now prove the Claim. The proof follows the lines of \cite[Lemma 4.1]{over}.\\
 Let us fix some notations. 
 We write $A(\epsilon_d, r) \lesssim B(\epsilon_d, r)$ if for any sequence $\epsilon_d$, there exists a constant $C>0$ and a set $E\subset H^0(\Sigma,\mathcal{F}^d\otimes \mathcal{E})^2$ of Gaussian measure smaller than $\exp(-C{\epsilon_d d})$ such that $A(\epsilon_d, r) \leq B(\epsilon_d, r)$ for any $(\alpha,\beta)\in H^0(\Sigma,\mathcal{F}^d\otimes \mathcal{E})^2\setminus E$ and any $r\in [1,3]$.\\
 Write $\alpha=f\cdot e_{\mathcal{F}}^d\otimes e_{\mathcal{E}}$ and $\beta=g\cdot e_{\mathcal{F}}^d\otimes e_{\mathcal{E}}$ so that $\alpha\otimes\nabla\beta-\beta\otimes\nabla\alpha=(fg'-gf')e_{\mathcal{F}}^{2d}\otimes e_{\mathcal{E}}^2\otimes\textrm{d}x$. Here $e_{\mathcal{F}}$ and $e_{\mathcal{E}}$ are local holomorphic trivializations of $\mathcal{F}$ and $\mathcal{E}$ over $U$.  In particular, this shows that the potential $\phi_d$  of the line bundle $\mathcal{F}^{2d}\otimes\mathcal{E}^2\otimes T^*_{\Sigma}$ is $2d\cdot\log\norm{e_{\mathcal{F}}}+O(1)=O(d)$. \\ 
 Finally, we will denote by $C_i$, for $i\in\mathbb{N}$, any constant which does \emph{not} depend on $\epsilon_d$ and $r$. \\
 
\textbf{Step 1:} We claim that 
\begin{equation}\label{claim1}
\int_{S(r)}\big|\log |fg'-gf'|\big|d\sigma_r\lesssim C_1d.
\end{equation}
We will use the identity $|\log t|=\log^+t+\log^-t$ and we treat separately  $\log^+$ and $\log^-$.\\
Write $\alpha=f\cdot e_{\mathcal{F}}^d\otimes e_{\mathcal{E}}$ and $\beta=g\cdot e_{\mathcal{F}}^d\otimes e_{\mathcal{E}}$, and then $\alpha\otimes\nabla\beta-\beta\otimes\nabla\alpha=(fg'-gf') e_{\mathcal{F}}^{2d}\otimes e_{\mathcal{E}}^2\otimes\textrm{d}x$, so that we have

\begin{equation}\label{poincare}
\log \norm{\alpha\otimes\nabla\beta-\beta\otimes\nabla\alpha}=\log |fg'-gf'|+\phi_d 
\end{equation}

Now, $\int_{S(r)}|\phi_d|d\sigma_r =O(d)$ and, by Proposition \ref{above}, we also have $\int_{S(r)}\log^+\norm{\alpha\otimes\nabla\beta-\beta\otimes\nabla\alpha}d\sigma_r\lesssim \epsilon d$ so that by (\ref{poincare}) we have

\begin{equation}\label{plus} \int_{S(r)}\log^+|fg'-gf'|d\sigma_r\leq \int_{S(r)}\log^+\norm{\alpha\otimes\nabla\beta-\beta\otimes\nabla\alpha}d\sigma_r+\int_{S(r)}|\phi_d|d\sigma_r\lesssim C_2 d.
\end{equation}
We now estimate the $\log^-$ part.
By Proposition \ref{below}, we know that $\log\norm{(\alpha\otimes\nabla\beta-\beta\otimes\nabla\alpha)(x_0)}\gtrsim -d$ and then, by (\ref{poincare}), we  get \begin{equation}\label{pointsub}
\log|(fg'-gf')(x_0)|\gtrsim -C_3 d.
\end{equation}
We denote by $P_r(x,z)=\frac{r-|x|^2}{|z-x|^2}$ the Poisson kernel on the ball of radius $r$. Using the identity $|\log t|=\log^+t+\log^-t$ and the fact that $\log |fg'-gf'|$ is subharmonic, we get
\begin{equation}\label{subha}
\log |(fg'-gf')(x_0)|+\int_{S(r)}P_r(x_0,z)\log^-|fg'-gf'(z)|d\sigma_r(z)\leq \int_{S(r)}P_r(x_0,z)\log^+|fg'-gf'(z)|\textrm{d}\sigma_r(z).
\end{equation}
By continuity of $P_r(x_0,z)$, we can find two positive constants $M,m$ such that $m\leq P_r(x_0,z)\leq M$ for any $|z|=r\in [\frac{1}{2},3]$. Then, by (\ref{subha}), we get
$$m\cdot\int_{S(r)}\log^-|fg'-gf'(z)|d\sigma_r(z)\leq M\cdot\int_{S(r)}\log^+|fg'-gf'(z)|d\sigma_r(z)-\log |(fg'-gf')(x_0)|.$$
Using the last inequality together with (\ref{plus}) and (\ref{pointsub}), we prove  (\ref{claim1}).\\

\textbf{Step 2:} We cover the unit circle $S(1)$ by a family of disjoint union of intervals $I^d_1,\dots,I^d_{q_d}$ of length smaller than $\epsilon_d^4$ and denote by $m^d_k=\sigma_1(I^d_k)$ the length of each interval so that $\sum_{k=1}^{q_d}m^d_k=1$. Take $r'$ such that $|r-\epsilon_d-r'|<\frac{1}{2}\epsilon_d^4$.
By Proposition \ref{below}, we can choose  points $x_k\in r'I^d_k$ such that, for any $k\in\{1,\dots,q_d\}$,  we have $\log\norm{\alpha\otimes\nabla\beta-\beta\otimes\nabla\alpha(x_k)}>-\epsilon_d d$, unless $(\alpha,\beta)$ lies in a set $E$ of measure smaller than $\sum_{k=1}^{q_d}e^{-C\epsilon_d d}$. Now, this measure is smaller than $e^{-C'\epsilon_d d}$, for some $C'<C$, as $q_d\sim \epsilon_d^{-4}$.\\
By the choice of $x_k$ and $I^d_k$ we have that if $z\in rI^d_k$ then $|z-x_k|<\epsilon_d+\frac{3}{2}\epsilon_d^4<2\epsilon_d$.
In particular we get 
\begin{equation}\label{rolle} \int_{S(r)}|\phi_d|\textrm{d}\sigma_r=\sum_{k=1}^q\int_{rI^d_k}|\phi_d|\textrm{d}\sigma_r\geq \sum_{k=1}^{q_d}m_k\phi_d(x_{k})-2\epsilon_d\sup |\textrm{d}\phi_d|
\end{equation}
where $\textrm{d}\phi_d$ is the differential of $\phi_d$, which is a $O(d)$. 
By the choice of $x_k$ and $I^d_k$ and since the function $\log|fg'-gf'(z)|$ is subharmonic, we can use \cite[Equation (29)]{over} to find  a positive $K>0$ (which does not depend  on $r$ and $\epsilon_d$) such that 
\begin{equation}\label{keyineq} 
\int_{S(r)}\log|fg'-gf'(z)|\textrm{d}\sigma_r(z)\geq\sum_{k=1}^{q_d}m_k\log|fg'-gf'(
x_{k})|-K\epsilon_d \int_{S(r)}\log|fg'-gf'(z)|\textrm{d}\sigma_r(z)
\end{equation}
Using first (\ref{poincare}) and then (\ref{rolle})-(\ref{keyineq}) we get
$$\int_{S(r)}\log\norm{\alpha\otimes\nabla\beta-\beta\otimes\nabla\alpha}\textrm{d}\sigma_r=\int_{S(r)}\log|fg'-gf'(z)|\textrm{d}\sigma_r+\int_{S(r)}\phi_d \textrm{d}\sigma_r$$
\begin{equation}\label{lastbefore}
\geq \sum_{k=1}^{q_d} m_k \log \norm{(\alpha\otimes\nabla\beta-\beta\otimes\nabla\alpha)(x_k)}-K\epsilon_d \int_{S(r)}\log|fg'-gf'(z)|\textrm{d}\sigma_r(z)- \epsilon_d \cdot \sup |\textrm{d}\phi_d|
\end{equation}
where $\textrm{d}\phi_d$ is the differential of $\phi_d$.
By (\ref{claim1}),  the choice of $x_k$ and (\ref{lastbefore}), we have that
\begin{equation}\label{finally}
\int_{S(r)}-\log\norm{\alpha\otimes\nabla\beta-\beta\otimes\nabla\alpha}\textrm{d}\sigma_r\lesssim \epsilon_d d +K\epsilon_d C_1 d+\epsilon_d\sup |\textrm{d}\phi_d| =C_4\epsilon_d d
\end{equation}
Using the identity $\log t=\log^+t-\log^-t$,  Equations (\ref{finally})  and Proposition  \ref{above}, we finally get
$$\int_{S(r)}\log^-\norm{W_{\alpha\beta}}\textrm{d}\sigma_r=\int_{S(r)}-\log\norm{W_{\alpha\beta}}\textrm{d}\sigma_r+\int_{S(r)}\log^+\norm{W_{\alpha\beta}}\textrm{d}\sigma_r\lesssim C_4\epsilon_d d+\epsilon_d d$$
which ends the proof of the Claim.
\end{proof}
\begin{oss}\label{notdep} Following the proof we can see that, for the case of  constant sequence $\epsilon_d\equiv \epsilon$,  the constant $C$ in the statement of Proposition \ref{fondame} is  independent of $\epsilon$.
\end{oss}
\section{Proofs of the main theorems}\label{prove}
In this section we will prove Theorems \ref{smooth}, \ref{number} and \ref{hole}. We will follow the notations of Section \ref{ramframwork}.
We fix a degree $1$ line bundle $\mathcal{F}$ over $\Sigma$, so that for any $\mathcal{L}\in\Pic^d(\Sigma)$ there exists an unique $\mathcal{E}\in \Pic_0(\Sigma)$ such that $\mathcal{L}=\mathcal{F}^d\otimes \mathcal{E}$. We denote by $u_{\alpha\beta}$ the branched covering defined by a pair  $(\alpha,\beta)$ of  global sections  of $\mathcal{F}^d\otimes \mathcal{E}$ without common zeros.
A \emph{critical point} of $u_{\alpha\beta}$ is a point $x\in\Sigma$ such that $\textrm{d}u_{\alpha\beta}(x)=0$. By Proposition \ref{bonnedef}, this is equivalent to the fact that that $x$ is a zero of the Wronskian $W_{\alpha\beta}\doteqdot\alpha\otimes\nabla\beta-\beta\otimes\nabla\alpha\in H^0(\Sigma;T^*_{\Sigma}\otimes \mathcal{F}^{2d}\otimes \mathcal{E})$, where $T^*_{\Sigma}$ is the cotangent bundle  of $\Sigma$.
For any pair  $(\alpha,\beta)$ of  global sections of $\mathcal{F}^d\otimes \mathcal{E}$, we denote by $$T_{\alpha\beta}=\frac{1}{2d+2g-2}\displaystyle\sum_{x\in\Crit(u_{\alpha\beta})}\delta_x$$  the empirical probability measure on the  critical points of $u_{\alpha\beta}$. Here, $\delta_x$ is the Dirac measure at $x$.
Finally, we will denote by $\norm{\cdot}$ any norm induced by the Hermitian metric $h$ on $\mathcal{L}$ given by Proposition \ref{metr}.

\begin{thm}\label{smoothfiber} Let $\Sigma$ be a  Riemann surface equipped with a  volume form $\omega$ of mass $1$ and $\mathcal{E}\in\Pic^0(\Sigma)$. For every smooth function  $f\in \mathcal{C}^{\infty}(\Sigma,\R)$, any degree $1$ line bundle $\mathcal{F}$ and any sequence $\epsilon_d\in\R_+$  of the form $\epsilon_d=O(d^{-a})$, for some $a\in [0,1)$, there exists a positive constant $C$ such that $$\mu_{\mathcal{F}^d\otimes\mathcal{E}}\bigg\{u\in\M_d(\Sigma,\mathcal{F}^d\otimes\mathcal{E}), \big|T_u(f)-\int_{\Sigma}f\omega\big|\geq\epsilon_d \bigg\}\leq \exp(-C\epsilon_dd).$$
\end{thm}
\begin{proof}
 We denote by $\omega_d$ the curvature form of $\mathcal{F}^{2d}\otimes \mathcal{E^2}^{2}\otimes T^*_{\Sigma}$ with respect to the (induced) metric given by Proposition \ref{metr}. Remark  that $\omega_d=2d\cdot\omega+O(1)$ so that
\begin{equation}\label{paragone}
\bigg\{ (\alpha,\beta)\in H^0(\Sigma,\mathcal{F}^d\otimes \mathcal{E})^2,\bigg| T_{\alpha\beta}(f)-\frac{1}{2d}\int_{\Sigma}f\omega_d \bigg|>\frac{\epsilon_d}{2} \bigg\}\supseteq\bigg\{(\alpha,\beta)\in H^0(\Sigma,\mathcal{F}^d\otimes \mathcal{E})^2,\bigg| T_{\alpha\beta}(f)-\int_{\Sigma}f\omega \bigg|>\epsilon_d \bigg\}.
\end{equation}
Remark that these sets are cones in $ H^0(\Sigma,\mathcal{F}^d\otimes \mathcal{E})^2$.  By Proposition \ref{vs}, this implies that the Gaussian measure of these sets equals the Fubini-Study measure of their projectivizations. In order to obtain the result, we will then compute the Gaussian measure of  the cones appearing in (\ref{paragone}).\\
By Poincar\'e-Lelong formula we have 
\begin{equation}\label{es} \bigg| T_{\alpha\beta}(f)-\frac{1}{2d}\int_{\Sigma}f\omega_d \bigg|=\frac{1}{2\pi d}\bigg|\int_{\Sigma}\log \norm{W_{\alpha\beta}}\partial\bar{\partial}f\bigg|
\leq \frac{\norm{\partial\bar{\partial}f}_{\infty}}{2 \pi d}\int_{\Sigma}\big|\log \norm{W_{\alpha\beta}}\big|\cdot\omega.
\end{equation}
The result then follows from the inequality (\ref{es}), the inclusion (\ref{paragone}) and  Proposition \ref{fondame}.
\end{proof}

\begin{proof}[Proof of Theorem \ref{smooth}]
We fix a degree $1$ line bundle $\mathcal{F}$ over $\Sigma$, so that for any $\mathcal{L}\in\Pic^d(\Sigma)$ there exists an unique $\mathcal{E}\in \Pic_0(\Sigma)$ such that $\mathcal{L}=\mathcal{F}^d\otimes \mathcal{E}$. The result then follows by integrating the inequality appearing in Theorem  \ref{smoothfiber} along the compact base $\Pic^0(\Sigma)\simeq \Pic^d(\Sigma)$ (the last isomorphism is given by the choice of the degree $1$ line bundle $\mathcal{F}$).
\end{proof}

\begin{oss}\label{constant}  Following the proof of Theorem \ref{smoothfiber} we see that we have prove a slight more precise result: for  any sequence $\epsilon_d\in\R_+$, there exists a positive constant $C$ such that for every smooth function  $f\in \mathcal{C}^{\infty}(\Sigma,\R)$ we have  $$\mu_d\bigg\{u\in\M_d(\Sigma), \big|T_u(f)-\int_{\Sigma}f\omega\big|>\epsilon_d \bigg\}\leq \exp(-C\frac{\epsilon_d}{\norm{\partial\bar{\partial} f}_{\infty}}d).$$
 Moreover, thanks to Remark \ref{notdep}, for the case of  constant sequence $\epsilon_d\equiv \epsilon$,  the constant $C$ in the statement  is also independent of $\epsilon$.
\end{oss}

\begin{proof}[Proof of Theorem \ref{number}]
Fix  $U\subset\Sigma$ an open set with piecewise $C^2$ boundary. Let $\psi_d^+,\psi_d^-$ be two families of  $\mathcal{C}^2$ functions such that \begin{itemize}
\item $0\leq \psi^-_d\leq\mathds{1}_U\leq \psi^+_d\leq 1$;
\item $ \frac{1}{2}\int_{\Sigma}\psi_d^-\omega\geq \Vol(U)-\frac{\epsilon_d}{2}$;
\item  $\frac{1}{2}\int_{\Sigma}\psi_d^+\omega\leq \Vol(U)+\frac{\epsilon_d}{2};$
\item $\norm{\partial\bar{\partial}\psi_d^+}_{\infty}=O\big(\frac{1}{\epsilon_d^2}\big)$ and  $\norm{\partial\bar{\partial}\psi_d^-}_{\infty}=O\big(\frac{1}{\epsilon_d^2}\big).$
\end{itemize}
These functions can be constructed as follows. Let $\rho: \R\rightarrow [0,1]$ be a smooth function such that $\rho(t)=1$ for $t\geq \frac{1}{3}$ and $\rho(t)=0$ for $t\geq \frac{2}{3}$. Then we define $\psi^+_d(x)=\rho(\frac{1}{\epsilon_d}\textrm{dist}(x,U))$ and $\psi^-_d(x)=1-\rho(\frac{1}{\epsilon_d}\textrm{dist}(x,\Sigma\setminus U))$, which are $\mathcal{C}^2$-functions thanks to the hypothesis on the boundary of $U$.\\ 
By Theorem \ref{smooth} for $f=\psi_d^+$ and by Remark \ref{constant}, there exists a constant $C_+>0$ and a set $E_{2}$ of measure smaller than $e^{-C_2\epsilon_d^3 d}$, such that for $u$ outside $E_+$, we have
$$\#(\Crit(u)\cap U)=T_u(\mathds{1}_U)\leq T_u(\psi_d^-)\leq 2d\int_{\Sigma}\psi_d^-\omega+\epsilon_d d \leq 2d\cdot\Vol(U)+2\epsilon_d d.$$
Using again  Theorem \ref{smooth} and Remark \ref{constant} for $f=\psi_d^-$, we can find $C_->0$ and a set $E_{-}$ of measure smaller than $e^{-C_1\epsilon_d^3 d}$, such that for $u$ outside $E_-$ we get $\#(\Crit(u)\cap U)\geq 2d\cdot\Vol(U)-2\epsilon_d d.$
This shows that, for any sequence $\epsilon_d$ and any $U\subset\Sigma$, there exists a positive constant $C$ (any constant smaller than $\min(C_1,C_2)$) and a set $E$ of measure smaller than $e^{-C\epsilon_d^3 d}$ (the union of $E_+$ and $E_-$), such that for $u$ outside $E$,  $\big|\frac{1}{2d}\#(\Crit(u)\cap U)-\Vol(U)\big|\leq\epsilon_d$,  which proves the theorem.
\end{proof}
\begin{proof}[Proof of Theorem \ref{hole}]
The proof follows the lines of the proof of Theorem \ref{number}. 
Fix  $U\subset\Sigma$ any open set. Let $\psi^+,\psi^-$ be two   smooth functions such that \begin{itemize}
\item $0\leq \psi^-\leq\mathds{1}_U\leq \psi^+\leq 1$;
\item $ \frac{1}{2}\int_{\Sigma}\psi^-\omega\geq \Vol(U)-\frac{\epsilon}{2}$;
\item  $\frac{1}{2}\int_{\Sigma}\psi^+\omega\leq \Vol(U)+\frac{\epsilon}{2};$
\end{itemize}
By Theorem \ref{smooth} for $f=\psi^+$ and by Remark \ref{constant}, there exists a constant $C_+>0$ and a set $E_{2}$ of measure smaller than $e^{-C_2 d}$, such that for $u$ outside $E_+$, we have
$$\#(\Crit(u)\cap U)=T_u(\mathds{1}_U)\leq T_u(\psi_2)\leq 2d\int_{\Sigma}\psi_2\omega+\epsilon d \leq 2d\cdot\Vol(U)+2\epsilon d.$$
Using again  Theorem \ref{smooth} and Remark \ref{constant} for $f=\psi^-$, we can find $C_->0$ and a set $E_{-}$ of measure smaller than $e^{-C_1d}$, such that for $u$ outside $E_-$ we get $\#(\Crit(u)\cap U)\geq 2d\cdot\Vol(U)-2\epsilon d.$
This shows that, for any  $\epsilon>0$ and any $U\subset\Sigma$, there exists a positive constant $C$ and a set $E$ of measure smaller than $e^{-C d}$, such that for $u$ outside $E$,  $\big|\frac{1}{2d}\#(\Crit(u)\cap U)-\Vol(U)\big|\leq\epsilon$. Taking $\epsilon=\Vol(U)$ we have the result.
\end{proof}
 \paragraph*{Acknowledgments.}
I would like to thank  Jean-Yves Welschinger for useful discussions. 
 This work was performed within the framework of the LABEX MILYON (ANR-10-LABX-0070)
of Universit\'e de Lyon, within the program "Investissements d'Avenir"
(ANR-11-IDEX-0007) operated by the French National Research Agency (ANR). 
%\cite{ke}

%\bibliographystyle{alpha}
\bibliographystyle{plain}
\bibliography{biblio}
\end{document}